\documentclass[2010]{amsart}
\usepackage{graphicx}
\usepackage{amsmath}
\usepackage{amsfonts}
\usepackage{amssymb}

\newtheorem{thm}{Theorem}[section]
\newtheorem{cor}[thm]{Corollary}
\newtheorem{lem}[thm]{Lemma}
\newtheorem{prop}[thm]{Proposition}
\theoremstyle{definition}
\newtheorem{defn}[thm]{Definition}

\newcommand{\supp}{\mathop{\rm supp}}

\theoremstyle{remark}

\numberwithin{equation}{section}
\newcommand{\norm}[1]{\left\Vert#1\right\Vert}
\newcommand{\abs}[1]{\left\vert#1\right\vert}

\newcommand{\Real}{\mathbb R}
\newcommand{\eps}{\varepsilon}
\newcommand{\To}{\longrightarrow}

\newcommand{\hs}{\mathcal{H}}
\newcommand{\sh}{\mathcal{S}}

\newcommand{\Shq}{\mathcal{S}_q(M^2_{\Lambda},L^2(\mu))}

\def\<{\langle}
\def\>{\rangle}

\begin{document}

\title[]{Embeddings of M\"{u}ntz spaces: the Hilbertian case}
\author{S.Waleed Noor}%
\address{Abdus Salam School of Mathematical Sciences, New Muslim Town, Lahore, Pakistan}
\email{waleed\_math@hotmail.com}%

\author{ Dan Timotin}

\address{Institute of Mathematics of the Romanian Academy, Calea Grivi\c tei 21, Bucharest, Romania}%

\email{Dan.Timotin@imar.ro}

\thanks{}%

\begin{abstract}
Given a strictly increasing sequence $\Lambda=(\lambda_n)$ of nonegative real numbers, with $\sum_{n=1}^\infty \frac{1}{\lambda_n}<\infty$, the M\"untz spaces $M_\Lambda^p$ are defined as the closure in $L^p([0,1])$ of the monomials $x^{\lambda_n}$.
We discuss properties of the embedding
$M_\Lambda^p\subset L^p(\mu)$, where $\mu$ is a finite positive Borel measure on the
interval $[0,1]$. Most of the results are obtained for the Hilbertian case $p=2$, in which we give conditions for the embedding to be bounded, compact, or to belong to the Schatten--von Neumann ideals.
\end{abstract}

\maketitle

\section{Introduction}

The M\"untz--Szasz Theorem states that,  if
$0=\lambda_0<\lambda_1<\dots<\lambda_n<\dots$ is an increasing sequence of
nonnegative real numbers, then the linear span of $x^{\lambda_n}$ is dense in
$C([0,1])$ if and only if $\sum_{n=1}^\infty \frac{1}{\lambda_n}=\infty$. When
$\sum_{n=1}^\infty \frac{1}{\lambda_n}<\infty$, the closed linear span of the
monomials $x^{\lambda_n}$ in different Banach spaces that contain them is
usually not equal to the whole space. In particular, if $1\leq p<+\infty$,  the
closed linear span of the monomials $x^{\lambda_n}$, $n\geq 0$, in  $L^p([0,1])$
is a proper subspace of $L^p([0,1])$. These spaces, called \emph{M\"untz spaces}
and denoted $M_\Lambda^p$, exhibit interesting properties that have not been
very much investigated. We refer principally to the monographies~\cite{Bor95,
Gur05}; recent results appear in~\cite{Alam08, Alam09, sp08, Dan10}.

Our starting point is the paper~\cite{Dan10},  which tries to find conditions
under which the space $M_\Lambda^1$ is continuously embedded in the Lebesgue
space $L^1(\mu)$. In full generality the problem is rather difficult; more
precise results are obtained for special classes of sequences $\Lambda$ and/or
measures $\mu$. One can see therein that, as a general rule,  the embedding
properties are related to the behavior of $\mu$ near the point 1, which for
these problems acts as a kind of ``distinguished boundary'' of the unit
interval.

The purpose of this paper is to investigate embedding results for other M\"untz spaces; we will actually focus on the Hilbert space $M^2_\Lambda$, although occasionally other values of $p$ also enter the picture. In the case of $M^2_\Lambda$ we will  also treat more refined properties of the embedding, namely its possible belonging to Schatten--von Neumann classes. As expected, the behavior of $\mu$ near the point 1 is again decisive. Examples are given to illustrate the difficulties.

The plan of the paper is the following. The next  section contains necessary preliminaries. In Section~\ref{se:general embeddings} we associate to a sequence $\Lambda$ a certain real valued function $\psi$; this function allows to formulate in a unitary manner some general embedding results. Sections~\ref{se:sublinear} and~\ref{se:schatten von neumann} focus on particular classes of measures and sequences, while in Section~\ref{se:interpolation} an interpolation result of Riesz--Thorin type is proved for embedding operators. Finally, Section~\ref{se:examples} presents two interesting examples.

\section{Preliminaries}\label{se:prelim}

\subsection{Riesz sequences and bases}

A sequence $\{f_k\}_{k=1}^\infty$ in a Hilbert space $H$ is said to be \emph{complete}
if span$\{f_k\}_{k\ge1}$ is dense in $H$. A \emph{Riesz sequence} for a separable Hilbert space $H$ is a  sequence
$\{f_k\}_{k=1}^\infty$ such that there exist a constant $C>0$ for which
\begin{equation}\frac{1}{C}\sum_{k=1}^m\abs{c_k}^2\leq\norm{\sum_{k=1}^mc_kf_k}^2\leq C\sum_{k=1}^m\abs{c_k}^2\end{equation}
for every finite scalar sequence $\{c_k\}_{k=1}^\infty$. A \emph{Riesz basis} is a complete Riesz sequence.

The following proposition gathers some well known facts about Riesz sequences (see, for instance,~\cite{Christen03})

\begin{lem}\label{le:gramian}
 Suppose $(x_n)$ is a sequence of vectors in the Hilbert space $\mathcal{F}$, and define its  \emph{Gramian} to be the infinite matrix $\Gamma=(\langle
x_n,x_m\rangle)_{n,m\in\mathbb{N}}$. Then:

(i) $(x_n)$ is a Riesz sequence if and only if $\Gamma$ defines an  invertible operator
on $\ell^2(\mathbb{N})$.

(ii) If $(g_n)_{n\in\mathbb{N}}$ is a Riesz
basis in the Hilbert $\mathcal{E}$, then $g_n\mapsto x_n$ can be extended to a
bounded
linear operator $J:\mathcal{E}\to\mathcal{F}$  if and only if $\Gamma$ defines a bounded operator
on $\ell^2(\mathbb{N})$; we have $\|J\|= \|\Gamma\|^{1/2}$.

(iii) The  sequence $(x_n)$ is a Riesz basis if and only if, for some (equivalently, any)  Riesz basis $(g_n)_{n\in\mathbb{N}}$  in the Hilbert $\mathcal{E}$, $g_n\mapsto x_n$ can be extended to an invertible
linear operator $J:\mathcal{E}\to\mathcal{F}$.
\end{lem}

\subsection{Schatten--Von Neumann classes} For $0<q<\infty$
the Schatten--Von Neumann class $\sh_q(\hs_1,\hs_2)$ is formed by the compact Hilbert space operators
$T:\hs_1\to\hs_2$ such that  $|T|=\sqrt{T^*T}:\hs_1\to\hs_1$ has a
family of eigenvalues $\{s_n(T)\}_{n=1}^\infty\in\ell_q$. If we define
\[
\|T\|_q= \left(
\sum_{n=1}^\infty s_n(T)^q
\right)^{1/q},
\]
then we obtain a quasinorm for $0<q<1$ and a norm for $q\ge 1$, with respect to which $\sh_q(\hs_1,\hs_2)$ is complete.
 It is immediate  that $\norm{T}_q\geq\norm{T}_{q'}$ for $q\leq q'$, hence $\sh_q\subset \sh_{q'}$. Again we gather in a lemma  some  properties that we will use; more information can be found, for instance, in~\cite{McCarthy67, Diestel95}. For the sake of this lemma, we will denote by $\|T\|_\infty$ the usual operator norm.

\begin{lem}\label{le:Schatten-vonNeumann}
Suppose $T':\hs_0\To\hs_1$, $T:\hs_1\To\hs_2$ are bounded operators, $0<q,q'\le\infty$.

(i)
If $T\in\sh_q(\hs_1,\hs_2)$, $T'\in \sh_{q'}(\hs_0,\hs_1)$ and
$1/r=1/q+1/q'$, then $T'T\in \sh_r$, and
\[
\|TT'\|_r\le \|T\|_{s}\|T'\|_{s'}.
\]

(ii) If $0<q\leq2$, then \[\|T\|_{\sh_q}^q=\inf\sum_n\|T\phi_n\|_{\hs_2}^q,\] where
the infimum is taken over all orthonormal bases $(\phi_n)_n$ of
$\hs_1$.

(iii) If $2\leq q<\infty$, then
\[\|T\|_{\sh_q}^q=\sup\sum_n\|T\phi_n\|_{\hs_2}^q\]
where the supremum is taken over all orthonormal bases
$(\phi_n)_n$ of $\hs_1$.
\end{lem}

Note that the right hand side in (i) and (ii) may be infinite, meaning that $T\not\in\sh_q$.   We will use the following corollary of Lemma~\ref{le:gramian} (iii) and Lemma~\ref{le:Schatten-vonNeumann} (ii).

\begin{cor}\label{co:Schatten von Neumann}
Suppose $T:\hs_1\To\hs_2$ is a bounded operator and $(x_n)$ is a Riesz basis in $\hs_1$.

(i) If $\frac{Tx_n}{\|Tx_n\|}$ is a Riesz sequence in $\hs_2$, then, for all $q>0$, $T\in \sh_q$ if and only if $(\|Tx_n\|)\in\ell^q$.

(ii) If
$0<q\le 2$, then
\[
\|T\|_{\sh_q}^q\le\sum_n\|Tx_n\|_{\hs_2}^q.
\]
In particular, $T\in\sh_q$ whenever the right hand side is finite for some Riesz basis $(x_n)$.
\end{cor}

The elements of $\sh_2(\hs_1,\hs_2)$
are called \emph{Hilbert--Schmidt operators} and by Lemma~\ref{le:Schatten-vonNeumann}
the \emph{Hilbert-Schmidt norm} of $T\in\sh_2(\hs_1,\hs_2)$ is
given by
$$\|T\|_{\sh_2}=(\sum_n \|T\phi_n\|_{\hs_2}^2)^{1/2}$$ where $(\phi_n)_n$
is any orthonormal basis of $\hs_1$.

\subsection{M\"{u}ntz spaces and embeddings}
We denote by $m$ the Lebesgue measure on $[0,1]$ and by $\norm{\cdot}_p$
the norm in $L^p(m)$ for $1\leq p\leq\infty$.

Suppose $\Lambda=\{\lambda_n\}_{n\geq1}$ an increasing sequence
of positive real numbers with $\sum^\infty_{n=1}\frac{1}{\lambda_n}<\infty$. As discussed in the introduction, the M\"untz space $M^p_\Lambda$ is defined to be the closure of the monomials $x^{\lambda_n}$, $n\ge1$ in $L^p(m)$; which is a proper  subspace of $L^p(m)$ by the M\"untz-Szasz Theorem. It is proven in~\cite{Bor95,Gur05} that the functions in $M^p_\Lambda$ are actually real analytic on the interval $(0,1)$ and continuous on $[0,1)$.

We will use the following two results concerning M\"untz spaces that appear in~\cite{Gur05}.

\begin{lem}[\cite{Gur05}, Corollary 8.1.2]\label{le:point estimate}
Any M\"{u}ntz polynomial $f(x)=\sum_{k=1}^m\alpha_kx^{\lambda_k}$ satisfies
$$\abs{f(x)}\leq 2(\sum_{k=1}^mx^{\lambda_k\beta_k})\norm{f}_\infty$$
for any $x\in[0,1]$ and any $\beta_k\geq0$ with $\sum_{k=1}^m\beta_k=1$.
\end{lem}

\begin{lem}[\cite{Gur05}, Proposition 8.2.2]\label{le:Bernstein}
There is a constant $K>0$ (depending only on $\Lambda$) such that, if $f(x)=\sum_{k=1}^m\alpha_kx^{\lambda_n}$, then
$$\norm{f'}_\infty\leq K(\sum_{k=1}^m\lambda_k)\norm{f}_\infty.$$
\end{lem}

For a fixed $p\ge1$, a positive measure $\mu$ on $[0,1]$ is
called  $\Lambda_p$ -\emph{embedding} if there is a constant
$C>0$ such that $$\norm{g}_{L^p(\mu)}\leq
C\norm{g}_p$$ for all polynomials $g\in
M^p_\Lambda$. Whenever $p$ is clear from the context,
we will remove the subscript $p$ and use the notation
$\Lambda$-embedding.

It follows easily from the definition (see~\cite{Dan10}) that a $\Lambda_p$-embedding measure $\mu$ has to satisfy $\mu({1})=0$. Therefore, as in Remark 2.5 of \cite{Dan10}, we may extend the embedding to all $f\in M^p_\Lambda$: if $\mu$ is $\Lambda_p$-embedding, then
$M^p_\Lambda\subset L^p(\mu)$ and $\norm{f}_{L^p(\mu)}\leq C\norm{f}_p$ for all $f\in M^p_\Lambda$.
For a $\Lambda_p$-embedding $\mu$ we denote by $i^p_\mu$ the embedding operator
$i^p_\mu:M^p_\Lambda\hookrightarrow L^p(\mu)$, which is bounded. If $0<\eps<1$, then the
interval $[1-\eps,1]$ will be denoted by $J_\eps$.

If $T:\mathcal{E}\to\mathcal{F}$ is a bounded operator on
Banach spaces, we define by
$\norm{T}_e=\inf_\mathcal{K}\|T+\mathcal{K}\|$ the
$\emph{essential norm}$ of an operator, where the infimum is
taken over all compact operators
$\mathcal{K}:\mathcal{E}\to\mathcal{F}$. This norm measures
how far an operator is from being compact. In
particular, $T$ is compact if and only if $\|T\|_e=0$. If $\mu$
is a positive measure on $[0,1]$, we will denote by $\mu_m$ the
measure equal to $\mu$ on $[0,1-\frac{1}{m}]$ and $0$
elsewhere, and $\mu_m'=\mu-\mu_m$.

The next proposition gathers some facts that are analogues of the corresponding results obtained in~\cite{Dan10} for $p=1$ and can be proved by adapting the methods therein; we state them without further comment.

\begin{prop}\label{pr:revision}
(i) If for some $\eps>0$ the restriction of $\mu$ to the interval $[1-\eps,1]$ is absolutely
continuous with respect to $m\mid_{[1-\eps,1]}$, with essentially bounded density, then $\mu$ is
$\Lambda$-embedding for any $\Lambda$.

(ii) If  $\supp\mu\subset[0,1-\eps]$, then $i_\mu$ is compact.

(iii) If $\mu$ is an embedding measure, then
\[\norm{i_\mu}_e=\lim_{n\rightarrow\infty}\norm{i_{\mu'_n}}.\]

(iv) Suppose there exists $\delta>0$ such that
$d\mu|_{J_\delta}=hdm|_{J_\delta}$ for some bounded measurable function $h$ with
$\lim_{t\rightarrow1}h(t)=a$. Then $i_\mu$ is bounded and $\norm{i_\mu}_e=a^{1/p}$.
\end{prop}

We will have the occasion to use the following elementary lemma.

\begin{lem}{(\cite{Dan10}, Lemma 2.2)}\label{le:elementary inequality}
Suppose $\rho:\Real_+\To\Real_+$ is an increasing, $\mathcal{C}^1$ function with $\rho(0)=0$ such that
$\mu(J_\eps)\leq\rho(\eps)$ for all $\eps\in(0,1]$. Then for any continuous, positive, increasing function
$g$ we have $$\int_{[0,1]}gd\mu\leq\int^1_0g(x)\rho'(1-x)dx.$$
\end{lem}

A sequence $\Lambda$ is \emph{lacunary} if for some $\gamma>1$ we have $\lambda_{n+1}/\lambda_n\geq \gamma, \ n\geq 1$.
The main feature of lacunarity is that the monomials $\lambda_n^{1/p}x^{\lambda_n}$ form a Riesz basis in each of the spaces $M^p_\Lambda$. In particular, the sequence $( \lambda_n^{1/2} x^{\lambda_n})_{n\ge1}$ forms a Riesz basis in $M^2_\Lambda$.

A more general class of sequences is defined as follows. The sequence $\Lambda$ is called  \emph{quasilacunary} if for some increasing sequence $\{n_k\}$ of
integers with $N:=\sup_k(n_{k+1}-n_k)<\infty$ and some $\gamma>1$ we have $\lambda_{n_{k+1}}/\lambda_{n_k}\geq \gamma$. It is easy to show that any quasilacunary sequence may be enlarged to one that is still quasilacunary and satisfies $\lambda_{n+1}/\lambda_n\leq \gamma^2$.  The main property of quasilacunary sequences is contained in the following lemma.

\begin{lem}{(\cite{Gur05}, Theorem 9.3.3)}\label{le:quasilacunary}
If $\Lambda$ is quasilacunary and $F_k=$span
$\{x^{\lambda_{n_k+1}},\ldots,\break x^{\lambda_{n_{k+1}}}\}$, then
$\exists$ $d_1,d_2>0$ such that for any sequence of functions
$f_k\in F_k$ we have
$$d_1(\sum_k\norm{f_k}^2_2)\leq\norm{\sum_k f_k}_2^2\leq
d_2(\sum_k\norm{f_k}^2_2).$$
\end{lem}

\section{Embeddings in $M_\Lambda^2$}\label{se:general embeddings}

Most of our results pertain to the Hilbert space $M_\Lambda^2$. In particular, in this section we will consider only $p=2$, and therefore we will drop the index $p$ and write ``$\Lambda$-embedding'' and ``$i_\mu$''. On the other hand, we will complicate things slightly by introducing  $M^2_{\Lambda,a}$ as the closure of the same monomials $x^{\lambda_n}$ in $L^2([0,a])$; so $M^2_\Lambda= M^2_{\Lambda,1}$.

It is known (see [B-E], pp. 177--178) that the condition $\sum_n
1/\lambda_n<\infty$ ensures that the system $x^{\lambda_n}$ is minimal in
$M^2_\Lambda$, and that, if $d_n$ is the distance from $x^{\lambda_n}$ to the
linear space  spanned by $x^{\lambda_m}$ with $m\not=n$,  then
$d_n=e^{-\gamma_n\lambda_n}$, with $\gamma_n\to 0$.
Let us denote
\begin{equation}\label{eq:definition of psi}
\psi(x)=\sum_{n\ge 1}d_n^{-1} x^{\lambda_n}.
\end{equation}
The remarks above show that the sum is convergent for any $x<1$, and defines an increasing function of $x$.

A simple argument of Hilbert space (also reproduced in [B-E], pp. 177--178) says that, if $p=\sum_i \alpha_i x^{\lambda_i}$, then
\[
|\alpha_n|\le d_n^{-1} \|p\|_2.
\]
It follows then that, for any $f\in M^2_\Lambda$, we have the estimate
\begin{equation}\label{eq:basic estimate}
|f^{(k)}(x)|\le \psi^{(k)}(x)\|f\|_2,\qquad k=0,1,\dots
\end{equation}
($f^{(k)}$ denoting, as usual, the $k$th derivative of $f$).

Consider now $0<a<1$. If $d_n(a)$ is the distance in $M^2_{\Lambda,a}$ from $x^{\lambda_n}$ to the space spanned by $x^{\lambda_m}$ with $m\not=n$, then
\[
d_n(a)^2=\inf_{p\in \hat{\mathcal{P}}_n} \int_0^a |x^{\lambda_n}-p(x)|^2\, dx=
\inf_{p\in \hat{\mathcal{P}}_n} \int_0^1 |a^{2\lambda_n} t^{\lambda_n}-p(t)|^2 a\, dt
=a^{2\lambda_n+1} d_n,
\]
whence
\[
\psi_a(x) :=\sum_{n\ge 1}d_n(a)^{-1} x^{\lambda_n}=a^{-1/2}\psi(a^{-1}x).
\]
We have thus the estimate, for functions in $M^2_{\Lambda,a}$,
\begin{equation}\label{eq:basic estimate a}
|f(x)|\le a^{-1/2}\psi(a^{-1}x) \|f\|_2.
\end{equation}
In particular, if $a=1$, we recapture~\eqref{eq:basic estimate} for $k=1$.

Although the function $\psi$ is a rather rough indicator of the properties of the sequences $\Lambda$,  it is useful in obtaining sufficient conditions for embedding results. A first example is an  analogue for $M^2_\Lambda$ of~\cite[Theorem~2.6]{Dan10}.

\begin{thm}\label{th:boundedness psi}
If $\psi\in L^2(\mu)$, then $\mu$ is $\Lambda$-embedding and $\|i_\mu \|\le\|\psi\|_{L^2(\mu)}$.
\end{thm}

\begin{proof}
The proof consists in integrating with respect to $\mu$
 the relation $|f(x)|\le \psi(x)\|f\|_2$, which, as noted above, is the case
$k=1$ of~\eqref{eq:basic estimate}.
\end{proof}

We obtain then the analogue for $M^2_\Lambda$ of~\cite[Corollary 2.7]{Dan10}.

\begin{cor}\label{co:boundedness rho}
Suppose $\rho:\mathbb{R}_+\to \mathbb{R}_+$ is an increasing $C^1$ function with $\rho(0)=0$, such that $\int_0^1 (\psi(x))^2\rho'(1-x)\, dx<\infty$. If $\mu(J_\epsilon)\le\rho(\epsilon)$ for all $\epsilon\in (0,1]$, then $\mu$ is $\Lambda$-embedding.
\end{cor}

\begin{proof}
Similar to the $L^1$ case, the proof follows from Theorem~\ref{th:boundedness psi} by applying Lemma~\ref{le:elementary inequality} to $g=\psi^2$.
\end{proof}

More interesting, we may improve Proposition~\ref{pr:revision} (ii): if the support of $\mu$ is compact in $[0,1)$, then the embedding is not only compact, but inside any Schatten--von Neumann class.

\begin{thm}\label{th:c_p embedding compact support}
If  $\supp\mu\subset[0,1-\eps]$, then $i_\mu\in \sh_q$ for any $q>0$.

\end{thm}

\begin{proof}
Denote $b=1-\epsilon$.
We have $\int \psi(x)^2\,d\mu(x)\le  \psi(b)^2 \|\mu\|$, and thus, by Theorem~\ref{th:boundedness psi},
\begin{equation}\label{eq:boundedness2}
\|i_\mu\|\le \psi(b)\sqrt{\|\mu\|}.
\end{equation}

Let us fix a positive integer~$k$ and a number $b<b'<1$. If $f\in M^2_\Lambda$, then, by~\eqref{eq:basic estimate}, we have
\[
\int_0^{b'} |f^{(k)}(x)|^2 \, dx\le (\psi^{(k)}(b'))^2 \|f\|^2
\]
and thus the $k$ times differentiation operator $D_k$ is bounded from $M^2_\Lambda$ to $L^2([0,b'])$, of norm at most $\psi^{(k)}(b')$.

On the other hand, integration is a Hilbert--Schmidt operator on $L^2([0,b'])$, of $\sh_2$ norm $b'/\sqrt{2}\le 2^{-1/2}$. It follows by Lemma~\ref{le:Schatten-vonNeumann} (i) that
$k$ times integration is an operator $J_k$ in $\sh_{2/k}$ on $L^2([0,b'])$, of norm at most $2^{-k/2}$. The composition $R_k=J_k\circ D_k$ is  the restriction of functions in $M^2_\Lambda$ to $[0,b']$; it is thus (again by Lemma~\ref{le:Schatten-vonNeumann} (i)) an operator  of class $\sh_{2/k}$ from $M^2_\Lambda$ to $L^2([0,b'])$, whose image is in $M^2_{\Lambda, b'}$, and
\[
\|R_k\|_{2/k}\le  2^{-k/2} \psi^{(k)}(b') .
\]

Consider then the embedding $i'_\mu$ from $M^2_{\Lambda, b'}$ into $L^2(\mu)$. According to~\eqref{eq:basic estimate a} and~\eqref{eq:boundedness2}, $i'_\mu$ is bounded and
\[
\|i'_\mu\|\le \psi_{b'}(b) \sqrt{\|\mu\|} = {b'}^{-1/2}\psi (b/b')\sqrt{\|\mu\|}.
\]
Finally, $i_\mu=i'_\mu R_k$, and thus
\begin{equation}\label{eq:estimate c_p}
\|i_\mu\|_{2/k}\le  2^{-k/2}  {b'}^{-1/2}\psi^{(k)}(b')\psi (b/b')\sqrt{\|\mu\|}.
\end{equation}
Choosing $k$ such that $2/k\le q$, inequality~\eqref{eq:estimate c_p} proves the theorem for any $q>0$.
\end{proof}

If $\mu$ is a general measure, Theorem~\ref{th:c_p embedding compact support} can still be used in order to obtain sufficient conditions for the embedding to be in~$\sh_2$. Namely, we  take a sequence $b_n\nearrow 1$ and define $\mu_j=\mu|[b_j, b_{j+1})$; then $i^2_\mu=\sum_j i^2_{\mu_j}$, and thus we have (for $q\ge 1$) $\|i^2_\mu\|_q\le \|\sum_j i^2_{\mu_j}\|_q$. We may then apply Theorem~\ref{th:c_p embedding compact support} to each of the measures $\mu_j$. The statements obtained  depend on the arbitrary sequence $(b_j)$, and are thus not very natural. We prefer to state a more elegant result, valid for Hilbert--Schmidt embeddings.

\begin{thm}
Define $\Psi(x)=\psi'(x^{1/4})\psi(x^{1/4})$. If $\Psi\in L^2(\mu)$, then $i_\mu\in \sh_{2}$.
\end{thm}

\begin{proof}
Consider the sequence $b_n$ defined by $b_0=0$, $b_1=1/2$, and $b_{j+1}=\sqrt{b_j}$; obviously it is an increasing sequence tending to~1. Define also, for $j\ge 0$, $\mu_j=\mu|[b_j, b_{j+1})$; we have
\begin{equation}\label{eq:direct sum mu}
L^2(\mu)=\bigoplus_{j=0}^\infty L^2(\mu_j).
\end{equation}
We have $\|i_\mu\|_2^2=\sum_{j=0}^\infty \|i_{\mu_j}\|_2^2$. By Theorem~\ref{th:c_p embedding compact support} applied with $b=b_{j+1}$ and $b'=b_{j+2}$, we have, for some constant $C>0$,
\[
\begin{split}
\|i_{\mu_j}\|_2^2&\le C(\psi'(b_{j+2}))^2 \left(
\psi \left(
\frac{b_{j+1}}{b_{j+2}}
\right)\right)^2 \|\mu_j\|
= C(\psi'(b_{j+2}))^2 (
\psi(b_{j+2}))^2 \|\mu_j\|\\&
=C(\Psi(b_j))^2 \|\mu_j\|.
\end{split}
\]
Since $\Psi$ is increasing the last term is less or equal $C\int\Psi(x)^2\, d\mu_j(x)$. Therefore
\[
\|i_\mu\|_2^2=\sum_{j=0}^\infty \|i_{\mu_j}\|_2^2\le  C \sum_{j=0}^\infty
\int\Psi(x)^2\, d\mu_j(x)= \int\Psi(x)^2\, d\mu(x),
\]
which proves the theorem.
\end{proof}

The function $\psi$ is hard to compute precisely, but one can give
estimates that can be used in the above embedding results. Here are some
examples; the estimates for the power series are classical and can be obtained,
for instance, by the techniques in~\cite{Br}.

\begin{enumerate}
\item Suppose $\lambda_n=2^n$ (a typical case of a lacunary sequence). Then $(2^{n/2}x^{2^n})$ forms a Riesz basis of $M^2_\Lambda$, so in~~\eqref{eq:definition of psi} we have $d_n\sim 2^{-n/2}$, and
\[
\psi(x)\sim \sum_{n=1}^\infty 2^{n/2}x^{2^n}\sim \frac{1}{\sqrt{1-x}}.
\]
This estimate shows the fact that the function $\psi$ reflects only partially
the properties of the sequence $\Lambda$. It
does not use the precise fact that $(2^{n/2}x^{2^n})$ is a Riesz basis, but only
that it is a uniformly minimal sequence. For instance, one cannot use it
to recapture the results obtained in Section~\ref{se:sublinear} below.

\item $\lambda_n=n^2$. This is a typical case of what is called a ``standard'' sequence (which is defined by the fact that $\lambda_{n+1}/\lambda_n\to 1$). A computation essentially done in~\cite[Section 7]{Dan10} shows that
\[
\psi(x)\sim e^{\frac{C}{1-x}}
\]
for some constant $C>0$.
\end{enumerate}

\section{Sublinear measures}\label{se:sublinear}

As in~\cite{Dan10}, one can obtain much more precise results if one considers special classes of measures.
In this section we will again consider different values of $p$, so we return to the notations $i_\mu^p$ and $M^p_\Lambda$.

\begin{defn}\label{de:sublinear}
A measure $\mu$ is called \emph{sublinear} if there is a constant $C>0$ such that for any $0<\eps<1$ we
have $\mu(J_\eps)\leq C\eps$. The smallest such $C$ will be denoted by $\norm{\mu}_S$. The measure $\mu$
is called \emph{vanishing sublinear} if $\lim_{\eps\rightarrow 0}\frac{\mu(J_\eps)}{\eps}=0$.
\end{defn}

As one can see, sublinear measures satisfy the condition of Corollary~\ref{co:boundedness rho} for $\rho(\epsilon)=\epsilon$.
The next lemma gathers some results that are either contained or analogues of~\cite[Section 4]{Dan10}.

\begin{lem}\label{le:sublinear quoted}
 (i)
Suppose that $\norm{\mu}_S<\infty$. If $g$ is continuous, positive, and increasing, we have
\[\int_{[0,1]}gd\mu\leq\norm{\mu}_S\int_{[0,1]}gdm.
\]

(ii) Suppose $1\le p <\infty$. If there exists $M>0$ such that $\frac{\lambda_{n+1}}{\lambda_n}\leq M$, then any $\Lambda_p$-embedding
measure is sublinear.

(iii)
If $\mu$ is sublinear, $1\le p <\infty$, then
\[
\sup_\lambda\norm{(p\lambda+1)^{1/p}x^\lambda}_{L^p(\mu)}\leq\norm{\mu}^{1/p}_S.
\]
\end{lem}

Precise results can be obtained for sublinear measures if we also assume that the sequence~$\Lambda$ is lacunary.

\begin{thm}\label{th:lacunary embedding 2}
Supposed $\Lambda$ is lacunary. If $\mu$ is a sublinear measure,
then $\mu$ is $\Lambda_2$-embedding, and $\|i_\mu^2\|$ is of order $\|\mu\|_S^{1/2}$.
\end{thm}

\begin{proof}Since $\Lambda$ is lacunary, the sequence of
functions $g_n=\lambda_n^{1/2}x^{\lambda_n}$ forms a Riesz basis
in $M^2_\Lambda$. The embedding $i_\mu^2:M^2_\Lambda\to L^2(\mu)$
is defined by $i_\mu^2(g_n)=g_n$.

We have $$\langle
g_n,g_m\rangle_{L^2(\mu)}=\lambda_n^{1/2}\lambda_m^{1/2}\int_{[0,1]}x^{\lambda_n+\lambda_m}d\mu.$$

Since $\mu$ is sublinear and the function
$x^{\lambda_n+\lambda_m}$ is continuous, positive, and
increasing, it follows from Lemma~\ref{le:elementary inequality} that
$$\int_{[0,1]}x^{\lambda_n+\lambda_m}d\mu\leq\|\mu\|_S\int_{[0,1]}x^{\lambda_n+\lambda_m}dm,$$
and thus
$$\langle g_n,g_m\rangle_{L^2(\mu)}\leq\|\mu\|_S\langle g_n,g_m\rangle_2.$$
Thus, if we define the matrices $A=(\langle g_n,g_m\rangle_{L^2(\mu)})$ and $B=(\|\mu\|_S\langle
g_n,g_m\rangle_2)$, then the entries of $A$ are nonegative and  majorized by those of $B$. But $B$ is bounded since
$(g_n)$ is a Riesz basis in $M^2_\Lambda$ (by Lemma~\ref{le:gramian} (i)); it follows that $A$
is also bounded. Therefore, by~\ref{le:gramian} (ii), the embedding $i_\mu^2$ is bounded, of norm of order $\|\mu\|_S^{1/2}$.
\end{proof}

Combining Lemma~\ref{le:sublinear quoted} (ii) with Theorem~\ref{th:lacunary embedding 2}, we obtain
\begin{cor}\label{co:necessary and sufficient embedding}
If $\Lambda$ is lacunary with
$\frac{\lambda_{n+1}}{\lambda_n}\leq M$ for some $M>0$, then a
measure $\mu$ is $\Lambda_2$-embedding if and only if it is
sublinear.
\end{cor}

Using Theorem~\ref{th:lacunary embedding 2}, we get vanishing sublinearity as
a sufficient condition for compactness of the embedding.

\begin{cor}\label{co:vanishing sublinear}
If $\Lambda$ is lacunary, then for any vanishing
sublinear measure $\mu$ the embedding $i_\mu:M^2_\Lambda\to L^2(\mu)$ is compact.
\end{cor}

\begin{proof}Recalling that $\mu_m$ is the measure equal to
$\mu$ on $[0,\frac{1}{m}]$ and equal to $0$ elsewhere on
$[0,1]$, we can view $i_{\mu_m}$ as the embeddings $M^2_\Lambda
\hookrightarrow L^2(\mu_m)$ and regard $L^2(\mu_m)$ as a
subspace of $L^2(\mu)$. Then $\mu_m'=\mu-\mu_m$ is the measure
$\mu$ restricted to $J_{\frac{1}{m}}$. By vanishing
sublinearity and Theorem~\ref{th:lacunary embedding 2}, $\mu_m'$ is $\Lambda$-embedding
with embedding constants $\|\mu_m'\|_S\to 0$ as $m\to\infty$. Therefore
$$\|i_\mu-i_{\mu_m}\|=\|i_{\mu_m'}\|=\|\mu_m'\|_S\To 0$$
as $m\to\infty$. Since $i_{\mu_m}$ is compact by Proposition~\ref{pr:revision} (ii), it follows that $i_\mu$ is compact.
\end{proof}

\section{Interpolation}\label{se:interpolation}

If $\Lambda$ is lacunary and $\mu$ is sublinear, then $\mu$ is $\Lambda_2$-embedding by Theorem~\ref{th:lacunary embedding 2}, while it is $\Lambda_1$-embedding by
Theorem 5.5 from~\cite{Dan10}. It is interesting that, although the M\"untz spaces do not form an interpolation scale of spaces, we may still apply the proof of the Riesz--Thorin theorem in order to extend the result to values $1<p<2$. We will actually obtain below a more general result concerning interpolation of embeddings.

\begin{thm}
Suppose $1\leq p_0<p_1<\infty$. For $0<t<1$, define $p_t$ by
$$\frac{1}{p_t}=\frac{1-t}{p_0}+\frac{t}{p_1}.$$
If a positive measure $\mu$ on $[0,1]$ is $\Lambda_{p_0}$-embedding and
$\Lambda_{p_1}$-embedding, then it is also $\Lambda_{p_t}$-embedding with
$$\|i_\mu\|_{p_t}\leq\|i_\mu\|_{p_0}^{1-t}\|i_\mu\|_{p_1}^{t}$$
where $\|i_\mu\|_{p_s}$ is the operator norm of
$i_\mu:M^{p_s}_\Lambda\to L^{p_s}(\mu)$ for $0\leq s\leq 1$.
\end{thm}

\begin{proof} The proof follows the Riesz--Thorin Theorem (see, for instance,~\cite{Bergh76}), so we will just sketch the main steps. For any $1\le p\le \infty$, denote, as customary, by $p'$ the conjugate exponent (satisfying $1/p+1/p'=1$).
Let $P_\Lambda$ be the space of all polynomials in
span$\{x^\lambda:\lambda\in\Lambda\}$. Then $P_\Lambda$ is dense
in $M^p_\Lambda$ and the theorem will be proved once we show that
$$\|f\|_{L^{p_t}(\mu)}\leq\|i_\mu\|_{p_0}^{1-t}\|i_\mu\|_{p_1}^{t}\|f\|_{p_t}$$
for all $f\in P_\Lambda$.

Fix then $f\in P_\Lambda$ with $\|f\|_{p_t}=1$. If $p_t'$ is the exponent conjugate
to $p_t$, then
\begin{equation}\label{eq:rth}
\|f\|_{L^{p_t}(\mu)}=\sup\{|\int_{[0,1]} fg\,d\mu|:g\text{ continuous}, \ \|g\|_{L^{p_t'}(\mu)}=1\}.
\end{equation}
Take then $g$ continuous with $\|g\|_{L^{p_t'}(\mu)}=1$. Define, for $0\le \Re z\le 1$,
\[
\frac{1}{p(z)}= \frac{1-z}{p_0}+\frac{z}{p_1},
\]
and, for $0\le x\le 1$,
$
\phi(z,x)=|f(x)|^{p_t/p(z)-1}f(x)$, $ \psi(z,x)=|g(x)|^{p'_t/p'(z)-1}g(x)$, $ F(z)=\int \phi(z,x)\psi(z,x)\,d\mu
$.
Then $F(z)$ is analytic on $0<\Re z<1$, bounded and continuous on $0\le\Re z\le 1$. From $\|f\|_{p_t}=1$ and $\|g\|_{L^{p_t'}(\mu)}=1$ it follows that
\[
\|\phi(is)\|_{p_0}= \|\phi(1+is)\|_{p_1}= \|\psi(is)\|_{L^{p_0'}(\mu)}= \|\psi(1+is)\|_{L^{p_1'}(\mu)}=1,
\]
which implies, by H\"older's inequality,
\[
|F(is)|\le \|\phi(is)\|_{L^{p_0(\mu)}}\|\psi(is)\|_{L^{p'_0(\mu)}}\le
\|i_\mu^{p_0}\| \|\phi(is)\|_{p_0}  \|\psi(is)\|_{L^{p_0'}(\mu)}=\|i_\mu^{p_0}\|,
\]
and similarly $|F(1+is|\le \|i_\mu^{p_1}\|$. The Three Lines Lemma yields then
\[
|\int f(x)g(x)\,d\mu(x)|= |F(t)|\le \|i^{p_0}_\mu\|^{1-t}\|i^{p_1}_\mu\|^{t},
\]
whence the theorem follows by~\eqref{eq:rth}.
\end{proof}

\begin{cor}\label{co:interpolation}
If $\Lambda$ is lacunary, then:

(i) every sublinear measure
$\mu$ is $\Lambda_p$-embedding for $1\leq p\leq 2$.

(ii) for any vanishing
sublinear measure $\mu$ the embedding $i^p_\mu:M^p_\Lambda\to L^p(\mu)$ is compact for $1\leq p\leq 2$.
\end{cor}

\section{Schatten--von Neumann embeddings}\label{se:schatten von neumann}

We return now to the case $p=2$ and the simplified notations of Section~\ref{se:general embeddings}. A slight strengthening of the sublinearity condition implies that for quasilacunary sequences the embedding belongs already to all Schatten--von Neumann classes.

\begin{thm}\label{th:quasilacunary schatten von neumann}
If $\Lambda$ is quasilacunary and the positive measure $\mu$ satisfies
$\mu(J_\eps)\leq C\eps^\alpha$ for some $\alpha>1$, then
$i_\mu\in \sh_q$ for all $q>0$.
\end{thm}

\begin{proof}
Again the  letter $C$ will be used for possibly different universal constants.
Suppose then that $\{n_k\}$ is a sequence of
integers with $N:=\sup_k(n_{k+1}-n_k)<\infty$, such that for  some $\gamma>1$ we have
 $\gamma\le\lambda_{n_{k+1}}/\lambda_{n_k+1}\le \gamma^{2(N-1)}$.
With the notations of Lemma~\ref{le:quasilacunary}, each subspace $F_k$ has dimension $n_{k+1}-n_k$. If we choose an orthonormal basis in each of the spaces $F_k$, it follows easily from Lemma~\ref{le:quasilacunary} that the union of these bases is a Riesz basis in $M^2_\Lambda$. Let us denote by $(\phi_i)$ this basis; we may also assume it  consists of real functions.


Suppose $\phi_i\in F_k$. Applying Lemma~\ref{le:point estimate}
to $\phi_i$ with $\beta_j=\beta=1/(n_{k+1}-n_k)$ and  Lemma~\ref{le:elementary inequality} for $g(x)=x^{2\beta\lambda_{n_k+1}}$ and $\rho(x)=Cx^{\alpha}$, we obtain
\[
\begin{split}
\norm{i_\mu\phi_i}^2_{L^2(\mu)}&=\int^1_0\abs{\phi_i(x)}^2d\mu(x)
\leq 4\norm{\phi_i}^2_\infty\int_0^1(\sum_{j=n_k+1}^{n_{k+1}}x^{\lambda_j\beta_j})^2d\mu(x)\\
&\hskip-1.2cm\leq4^{(N+1)}\norm{\phi_i}^2_\infty\sum_{j=n_k+1}^{n_{k+1}}\int_0^1x^{2\lambda_j\beta_j}d\mu(x)
\leq4^{(N+1)}\norm{\phi_i}^2_\infty N\int_0^1x^{2\beta\lambda_{n_k+1}}d\mu(x)\\
&\hskip-1.2cm\leq4^{(N+1)}\norm{\phi_i}^2_\infty N\int_0^1x^{2\beta\lambda_{n_k+1}}(\alpha C(1-x)^{\alpha-1})dx
\leq C\norm{\phi_i}^2_\infty B(2\beta\lambda_{n_k+1}+1,\alpha).
\end{split}
\]
The Euler beta function satisfies the following asymptotic formula:
if $t$ is large and $s$ is fixed, then
\begin{equation*}
B(t,s)=\int^1_0x^{t-1}(1-x)^{s-1}dx\sim \Gamma(s)t^{-s}.
\end{equation*}
Therefore
\begin{equation}\label{eq:zero estimate}
\norm{i_\mu\phi_i}^2_{L^2(\mu)} \leq C \norm{\phi_i}^2_\infty \Gamma(\alpha)(2\beta\lambda_{n_k+1}+1)^{-\alpha}
\leq C\frac{\norm{\phi_i}^2_\infty }{\lambda_{n_k+1}^\alpha}
.
\end{equation}

It is elementary to see (and can be found in~\cite{Dan10}, Lemma 5.4) that
if $f:[0,1]\rightarrow \Real$ is a nonconstant differentiable function, then
$$\norm{f}_1\geq \min\{\frac{\norm{f}_\infty}{4},\frac{\norm{f}^2_\infty}{2\norm{f'}_\infty}\}.$$
We apply this inequality to  $f=\phi_i^2$. If the minimum is given by the first term above, then $\|\phi_i\|_\infty^2\le 4 \|\phi_i\|^2_2=4$, whence from~\eqref{eq:zero estimate} it follows that
\begin{equation}\label{eq:first estimate}
\norm{i_\mu\phi_i}_{L^2(\mu)}\lesssim \frac{4}{\lambda_{n_k}^{\alpha/2}} \lesssim\frac{1}{\gamma^{\alpha(k-1)/2}}.\end{equation}

If the minimum  is given by the second term, that is, $\norm{f}^2_\infty\leq 2\norm{f}_1\norm{f'}_\infty$, noting that $\norm{f}_\infty=\norm{\phi_i}^2_\infty$, $\norm{g}_1=\norm{\phi_i}^2_2=1$, and $f'=2\phi_i\phi_i'$, we obtain, using Lemma~\ref{le:Bernstein},
 \[
 \norm{\phi_i}^{4}_\infty\leq 4\norm{\phi_i}^2_2\norm{\phi_i'}_\infty
\norm{\phi_i}_\infty\leq 4KN\lambda_{n_{k+1}}\norm{\phi_i}^2_\infty.
\]
Therefore $\norm{\phi_i}_\infty
\leq C\lambda_{n_{k+1}}^{1/2}$, whence, again by~\eqref{eq:zero estimate},
\[
\norm{i_\mu\phi_i}_{L^2(\mu)}\leq C\frac{\norm{\phi_i}_\infty}{\lambda_{n_k+1}^{\alpha/2}}
\leq C\frac{\lambda_{n_{k+1}}^{1/2}}{\lambda_{n_k+1}^{\alpha/2}}\leq C \Big(\frac{\lambda_{n_{k+1}}}{\lambda_{n_k+1}}\Big)^{1/2}\frac{1}{{\lambda_{n_k+1}}^{(\alpha-1)/2}}.\]
The factor $(\lambda_{n_{k+1}}/\lambda_{n_k+1})^{1/2}$ is bounded by $\gamma^{N-1}$, and therefore
\begin{equation}\label{eq:second estimate}
\norm{i_\mu\phi_i}_{L^2(\mu)}\leq C\frac{1}{\lambda_{n_k}^{\frac{\alpha-1}{2}}}
\leq C\frac{1}{\gamma^{\frac{(\alpha-1)(k-1)}{2}}}.
\end{equation}

Note that the inequalities ~\eqref{eq:first estimate} and ~\eqref{eq:second estimate} have been obtained for
$ i=n_k+1,\ldots,n_{k+1}$. They imply, if $q\le 2$, that
\[
\begin{split}
\sum_{i=1}^\infty\norm{i_\mu \phi_i}^q_{L^2(\mu)}
&=\sum_{k=0}^\infty\sum_{i=n_k+1}^{n_{k+1}}\norm{i_\mu \phi_i}^q_{L^2(\mu)}\\
&\leq C
N\left(\sum_{k=0}^\infty\frac{1}{(\gamma^{\alpha q/2})^{k-1}}+\sum_{k=0}^\infty\frac{1}{(\gamma^{(\alpha-1)q/2})^{k-1}}\right)<\infty.
\end{split}
\]
By Corollary~\ref{co:Schatten von Neumann} (ii), it follows that
 $i_\mu \in \Shq$ for all $q\leq 2$, and therefore for all $q>0$.\end{proof}

In particular, for $\Lambda$ lacunary the condition $\mu(J_\eps)\leq C\eps^\alpha$ for some $\alpha>1$ implies that the embedding is in all Schatten--von Neumann classes.



\section{Examples}\label{se:examples}

\subsection{}
In the first example we intend to construct a measure $\mu$ and a sequence $\Lambda$ such that $\mu$ is $\Lambda_p$-embedding for
$p=2$ but not for $p=1$. As above, we will use the same letter $C$ for possibly different universal constants.

Take $\mu=\sum_kc_k\delta_{a_k}$, with $0<a_k<1$. We will define recurrently $\lambda_n\to\infty$, $a_n\to 1$, and $c_n\to 0$
such as to have:

(A) \ $\sup_n\lambda_nc_na_n^{\lambda_n}=\infty$;

(B) \ $\sum_k\lambda_nc_ka_k^{2\lambda_n}\leq C \frac{\ln n}{n^2}$.

\smallskip

First, one can start with $\lambda_1=1$, $a_1=1/2$ and $c_1=1$. Suppose then that $\lambda_k,a_k,c_k$ have been obtained for $k\leq n-1$. Choose first $\lambda_n$ sufficiently large such that

(i) \ $\lambda_n\sum_{k\leq n-1}a_k^{\lambda_n}\leq \frac{1}{n^2}$;

(ii) \ $\lambda_{n+1}\geq n^4\lambda_n$.

Put $a_n=1-\frac{2\ln n}{\lambda_n}$ and $c_n=\frac{2n^2\ln n}{\lambda_n}$. Then
\begin{equation}\label{eq:10}
a_n^{\lambda_n}=(1-\frac{2\ln n}{\lambda_n})^{\lambda_n}\sim e^{-2\ln n}=\frac{1}{n^2},\end{equation}
whence
$\lambda_nc_na_n^{\lambda_n}\sim\ln n$
and thus (A) is satisfied.

To achieve (B),  write
\begin{equation}\label{eq:11}
\sum_k\lambda_nc_ka_k^{2\lambda_n}=\sum_{k\leq n-1}\lambda_nc_ka_k^{2\lambda_n}
+\lambda_nc_na_n^{2\lambda_n}+\sum_{k\geq n+1}\lambda_nc_ka_k^{2\lambda_n}.\end{equation}
The first sum is smaller than $\frac{1}{n^2}$ by (i). The second term is of order $\frac{\ln n}{n^2}$
by~\eqref{eq:10} and ~\eqref{eq:11}. For the third term, we have $$\sum_{k\geq n+1}\lambda_nc_ka_k^{2\lambda_n}
\leq\lambda_n\sum_{k\geq n+1}c_k.$$
From (ii) it follows in particular that $c_n$ decreases faster than a geometric progression
(which also proves the convergence of the sum defining $\mu$), and thus for some constant $C$ we have
$$\sum_{k\geq n+1}c_k\leq Cc_{n+1}=\frac{C2(n+1)^2\ln(n+1)}{\lambda_{n+1}}.$$
Applying again (ii), $$\lambda_n\sum_{k\geq n+1}c_k\leq C\frac{\ln n}{n^2}.$$
So we have estimated all three terms of (3.3) by $\frac{\ln n}{n^2}$, whence (B) is satisfied.

Now (ii) implies that $\Lambda$ is lacunary and the functions $g_k(x)=\lambda_k^{1/2}x^{\lambda_k}$ form a Riesz basis in $M^2_\Lambda$. Thus
\begin{equation}\label{eq:21}
\|\sum_kb_kg_k\|^2_2\sim\sum_k\abs{b_k}^2.\end{equation}

On the other hand, \begin{equation}\label{eq:22}
\|\sum_kb_kg_k\|^2_{L^2(\mu)}\leq(\sum_k\abs{b_k}\norm{g_k}_{L^2(\mu)})^2
\leq(\sum_k\abs{b_k}^2)(\sum_k\norm{g_k}^2_{L^2(\mu)}).\end{equation}
According to (B), we have $$\norm{g_n}^2_{L^2(\mu)}=\sum_kc_k\lambda_na_k^{2\lambda_n}\leq C\frac{\ln n}{n^2},$$
and thus $\sum_k\norm{g_k}^2_{L^2(\mu)}<\infty$.
So it follows from~\eqref{eq:21} and~\eqref{eq:22} that
$$\|\sum_kb_kg_k\|^2_{L^2(\mu)}\leq C (\sum_k\norm{g_k}^2_{L^2(\mu)})(\|\sum_kb_kg_k\|^2_2).$$
Thus $\mu$ is $\Lambda_2$-embedding.

On the other side, for $p=1$
$$\norm{\lambda_nx^{\lambda_n}}_{L^1(\mu)}=\int_{[0,1]}\lambda_nx^{\lambda_n}d\mu=\sum_kc_k\lambda_na_k^{\lambda_n}
\geq c_n\lambda_na_n^{\lambda_n}.$$
So by (A) $$\sup_n\norm{\lambda_nx^{\lambda_n}}_{L^1(\mu)}\geq\sup_n c_n\lambda_na_n^{\lambda_n}=\infty,$$
whence $\norm{\lambda_nx^{\lambda_n}}_1\leq 1$ for all $n=1,2,\ldots$.
Hence $\mu$ is not $\Lambda_1$-embedding.

\subsection{}

In this example we consider the Hilbert space $M^2_\Lambda$. We will show that for
any $0<r<q$ we can construct a lacunary sequence $\Lambda$
and a measure $\mu$  such that
 $i^2_\mu\notin \sh_r$ but $i^2_\mu\in \sh_q$.

Fix $q>r>0$. Choose a sequence
$\{\alpha_n\}\in\ell^q$ with $\abs{\alpha_n}<1$ but $\{\alpha_n\}\notin\ell^r$, and a double sequence $\{\beta_{nm}\}_{n,m=1}^\infty$
with $\sum_n\sum_m\beta_{nm}<\frac{1}{4}$.

The measure will again be of the form $\mu=\sum_j c_j\delta_{a_j}$, with $0<a_j<1$; we will construct  $a_n,c_n$ as well as $\lambda_n$ recurrently.
Suppose $a_j,c_j,\lambda_j$ have been obtained for $j\leq n-1$.  Choose first $\lambda_n$ sufficiently large such that $\Lambda$ is lacunary,
and
\begin{align}
\sum_{i=1}^{n-1}c_i\lambda_na_i^{2\lambda_n}&\leq\frac{1}{8}\alpha_n^2,\label{eq:72}\\
\sum_{i=1}^{n-1}c_i\lambda_j^{1/2}\lambda_n^{1/2}
a_i^{\lambda_j+\lambda_n}&\leq\frac{1}{4}\alpha_j\alpha_n\beta_{jn}^{1/2}\quad\text{for $j=1,\ldots,n-1$},\label{eq:73}\\
\frac{\alpha_n^2\lambda_i^{1/2}\lambda_j^{1/2}}
{\lambda_n}&\leq\frac{1}{2^{n+2-\max\{i,j\}}}
\alpha_i\alpha_j\beta_{ij}^{1/2}\quad\text{for $i,j<n$},\label{eq:74}\\
\frac{\alpha_n^2\lambda_i^{1/2}}{\lambda_n^{1/2}}&\leq
\frac{1}{2}\alpha_i\alpha_n\beta_{in}^{1/2}\quad\text{for $i,j<n$}.\label{eq:75}
\end{align}

Take then $a_n=e^{-1/2\lambda_n}$ and $c_n=\frac{\alpha_n^2}{\lambda_n}$; we have then
$c_n\lambda_na_n^{2\lambda_n}=\frac{\alpha_n^2}{e}$.

Consider then the Riesz basis $(g_n)=(\lambda_n^{1/2}x^{\lambda_n})$ for $M^2_\Lambda$. By~\eqref{eq:72} and~\eqref{eq:74}, we have
 \[
 \begin{split}
 \norm{i^2_\mu g_n}_{L^2(\mu)}^2&=\int_{[0,1]}\lambda_n x^{2\lambda_n}d\mu=\sum_jc_j\lambda_n a_j^{2\lambda_n}
\\ &=\sum_{j\leq n-1}c_j\lambda_na_j^{2\lambda_n}+c_n\lambda_n a_n^{2\lambda_n}+\sum_{j\geq n+1}c_j
\lambda_n a_j^{2\lambda_n}\\
&\leq\frac{1}{8}\alpha_n^2+\alpha_n^2+\sum_{j\geq n+1}\frac{\alpha_j^2\lambda_na_j^{2\lambda_n}}{\lambda_j}\leq\frac{1}{8}\alpha_n^2+\alpha_n^2+\sum_{j\geq n+1}\frac{1}{2^{j+2-n}}\alpha_n^2\beta_{nn}^{1/2}\\
&\leq\frac{1}{8}\alpha_n^2+\alpha_n^2+\frac{1}{4}\alpha_n^2<\frac{3}{2}\alpha_n^2.
 \end{split}
  \]
Since clearly $\norm{i^2_\mu g_n}_{L^2(\mu)}^2\ge c_n\lambda_na_n^{2\lambda_n}$  we obtain
 \begin{equation}\label{eq:77}\frac{1}{e}\alpha_n^2
 \le\norm{i^2_\mu g_n}_{L^2(\mu)}^2\le\frac{3}{2}\alpha_n^2.\end{equation}

If we define $f_n=i^2_\mu g_n/\norm{i^2_\mu g_n}_{L^2(\mu)}$, then $f_n$ are nonnegative functions, and~\eqref{eq:77} implies that $\langle f_n,f_m\rangle_{L^2(\mu)}\le e\alpha_n^{-1}\alpha_m^{-1}\lambda_n^{1/2}\lambda_m^{1/2}\int_{[0,1]}x^{\lambda_n+\lambda_m}d\mu(x)$. The last quantity can be estimated
 using~\eqref{eq:73}, \eqref{eq:74}, and \eqref{eq:75};
 we have, for $n>m$,
\[
\begin{split}
&\hskip-.5cm e\alpha_n^{-1}\alpha_m^{-1}\lambda_n^{1/2}\lambda_m^{1/2}\int_{[0,1]}x^{\lambda_n+\lambda_m}d\mu(x)=e\alpha_n^{-1}\alpha_m^{-1}\lambda_n^{1/2}\lambda_m^{1/2}\sum_jc_ja_j^{\lambda_n+\lambda_m}\\
&=e\alpha_n^{-1}\alpha_m^{-1}(\sum_{j\leq n-1}c_j\lambda_n^{1/2}\lambda_m^{1/2}a_j^{\lambda_n+\lambda_m}
+c_n\lambda_n^{1/2}\lambda_m^{1/2}a_n^{\lambda_n+\lambda_m}
+\sum_{j\geq n+1}c_j\lambda_n^{1/2}\lambda_m^{1/2}a_j^{\lambda_n+\lambda_m})\\
&\leq e\alpha_n^{-1}\alpha_m^{-1}(\frac{1}{4}\alpha_n\alpha_m\beta_{nm}^{1/2}+
\frac{\alpha_n^2}{\lambda_n}\lambda_n^{1/2}\lambda_m^{1/2}a_n^{\lambda_n+\lambda_m}
+\sum_{j\geq n+1}\frac{\alpha_j^2}{\lambda_j}\lambda_n^{1/2}\lambda_m^{1/2}a_j^{\lambda_n+\lambda_m})\\
&\leq e\alpha_n^{-1}\alpha_m^{-1}(\frac{1}{4}\alpha_n\alpha_m\beta_{nm}^{1/2}+
\frac{\alpha_n^2\lambda_m^{1/2}}{\lambda_n^{1/2}}
+\sum_{j\geq n+1}\frac{\alpha_j^2\lambda_n^{1/2}\lambda_m^{1/2}}{\lambda_j})\\
&\leq e\alpha_n^{-1}\alpha_m^{-1}(\frac{1}{4}\alpha_n\alpha_m\beta_{nm}^{1/2}+
\frac{1}{2}\alpha_m\alpha_n\beta_{mn}^{1/2}
+\sum_{j\geq n+1}\frac{1}{2^{j+2-n}}\alpha_m\alpha_n\beta_{mn}^{1/2})\\
&\leq\frac{e}{4}\beta_{nm}^{1/2}+\frac{e}{2}\beta_{nm}^{1/2}+\frac{e}{4}\beta_{nm}^{1/2}=e\beta_{nm}^{1/2}.
\end{split}
\]

Therefore
\begin{equation}\label{eq:78}
\sum_{n\neq m}\abs{\langle f_n,f_m\rangle_{L^2(\mu)}}^2=2\sum_{n>m}\langle f_n,f_m
\rangle_{L^2(\mu)}^2\leq e\sum_{n,m}\beta_{nm}<\frac{e}{4}.
\end{equation}

Consider the Gramian $\Gamma=\{\langle f_n,f_m\rangle\}_{n,m=1}^\infty$ of the sequence $(f_n)$. If we define $\Gamma_0=\Gamma-I$, then
\[
\begin{split}
\norm{\Gamma_0}_{\sh_2}^2&=\sum_n\norm{\Gamma_0 e_n}_{\ell^2}^2
=\sum_{n,m}\abs{\langle \Gamma_0e_n,e_m\rangle_{\ell^2}}^2\\&=\sum_{n\neq m}\abs{\langle
\Gamma e_n,e_m\rangle_{\ell^2}}^2
=\sum_{n\neq m}\abs{\langle f_n,f_m\rangle_{L^2(\mu)}}^2<\frac{e}{4}
\end{split}
\]
by~\eqref{eq:78}. Therefore $\norm{\Gamma_0}\leq\norm{\Gamma_0}_{\sh_2}<\sqrt{e}/2$, whence $\Gamma$ is invertible. This implies by Lemma~\ref{le:gramian}~(i) that $(f_n)$ is a Riesz sequence in $L^2(\mu)$.

By~\eqref{eq:77} and the choice of $\alpha_n$, the sequence $(\norm{i^2_\mu g_n})$ is in $\ell_q$ but not in $\ell_r$. Corollary~\ref{co:Schatten von Neumann}~(i) implies then that $i^2_\mu$ is in $\sh_q$ but not in $\sh_r$, as desired.

\section{Final remarks}

It is often the case when dealing with M\"untz spaces that  results that are
valid for lacunary sequences can be extended, albeit sometimes after significant
work, to the quasilacunary case. We have already seen such an situation
in Section~\ref{se:schatten von neumann}, where the proof of Theorem~\ref{th:quasilacunary schatten von neumann} would be actually simpler if one assumes $\Lambda$ lacunary.
In particular, the continuity  of the embedding $i^1_\mu$ for sublinear measures
is shown in~\cite{Dan10} for quasilacunary sequences. However, a similar result
is not yet proved for $i^2_\mu$; at least the proof of our
Theorem~\ref{th:lacunary embedding 2} does not seem to extend easily to
quasilacunary sequences. It is an open problem to provide such an extension.

Another interesting open question is the possible extension of Corollary~\ref{co:interpolation} to the range $2<p<\infty$; in particular, is it true in that case that, if $\Lambda$ is lacunary, then any sublinear measure is $\Lambda_p$-embedding?

Finally, let us note that an important application of embedding theorems is the study of boundedness properties of composition and multiplication operators with domain M\"untz spaces (see~\cite{Alam08, Alam09, Dan10}). We will give such applications in a forthcoming paper.


\end{document}